\documentclass[11pt,a4paper,twoside]{article}
\usepackage[hmarginratio=1:1,bmargin=1.3in]{geometry}
\usepackage[OT2,T1]{fontenc}
\usepackage{amsmath, amsthm}
\usepackage{amsfonts}
\usepackage{amssymb}
\usepackage[all]{xy}
\usepackage{graphicx}
\usepackage{color}
\usepackage[nottoc, notlof, notlot]{tocbibind}
\usepackage{array}
\usepackage{url}
\usepackage{rotating}
\usepackage{fancyhdr}
\usepackage{fancyhdr}
\usepackage{titlesec}
\usepackage{xfrac}
\usepackage{multicol}
\usepackage{comment}
\usepackage{diagbox}
\usepackage{pifont}
\usepackage{multirow}
\usepackage{enumerate}
\newcommand{\xmark}{\ding{55}}%

\titleformat{\section}[hang]{\center\Large\bf}{\thesection.}{0.5cm}{}

\DeclareSymbolFont{cyrletters}{OT2}{wncyr}{m}{n}
\DeclareMathSymbol{\Sha}{\mathalpha}{cyrletters}{"58}
\DeclareMathSymbol{\Brusse}{\mathalpha}{cyrletters}{"42}

\theoremstyle{plain}
\newtheorem{thm}{Theorem}[section]

\newtheorem{lem}[thm]{Lemma}
\newtheorem{prop}[thm]{Proposition}

\theoremstyle{definition}

\newtheorem{rem}[thm]{Remark}

\newtheorem{exa}[thm]{Example}

\newtheorem{ques}[thm]{Question}

\newcommand{\N}{\mathbb{N}}
\newcommand{\Z}{\mathbb{Z}}

\newcommand{\id}{\mathrm{id}}

\renewcommand{\hom}{\mathrm{Hom}}
\newcommand{\ext}{\mathrm{Ext}}

\newcommand{\aut}{\mathrm{Aut}}

\newcommand{\out}{\mathrm{Out}}
\newcommand{\inn}{\mathrm{Inn}}
\newcommand{\gm}{\mathbb{G}_{\mathrm{m}}}
\newcommand{\ga}{\mathbb{G}_{\mathrm{a}}}
\newcommand{\br}{\mathrm{Br}}

\newcommand{\sln}{\mathrm{SL}_n}
\newcommand{\pgl}{\mathrm{PGL}}
\newcommand{\gln}{\mathrm{GL}_n}

\newcommand{\pic}{\mathrm{Pic}}

\newcommand{\spec}{\mathrm{Spec}}
\renewcommand{\cal}[1]{\mathcal{#1}}
\newcommand{\bb}[1]{\mathbb{#1}}

\newcommand{\Sym}{\mathrm{Sym}}

\usepackage{marvosym}

\title{On composition of torsors}
\author{Mathieu \textsc{Florence}, Diego \textsc{Izquierdo} and Giancarlo \textsc{Lucchini Arteche}}
\date{}

\begin{document}

\maketitle

\begin{abstract} Let $K$ be a field, let $X$ be a connected smooth $K$-scheme and let $G,H$ be two smooth connected $K$-group schemes. Given $Y \to X$ a $G$-torsor and $Z \to Y$ an $H$-torsor, we study whether one can find an extension $E$ of $G$ by $H$ so that the composite $Z \to X$ is an $E$-torsor. We give both positive and negative results, depending on the nature of the groups $G$ and $H$.\\

\end{abstract}

\section{Introduction}

Consider a field $K$, a smooth\footnote{In this article, we follow Hartshorne's definition of smoothness, which in particular implies that the scheme is of finite type (this is \emph{not} an assumption taken by the Stacks Project, for instance).} connected $K$-scheme $X$ and two smooth connected $K$-group schemes $G$ and $H$. In the present article, we are interested in the following question about compositions of torsors:

\begin{ques} \label{ques} Let $Y \to X$ be a $G$-torsor, and let $Z \to Y$ be an $H$-torsor. Can one find an extension of $K$-group schemes
\[1\to H\to E\to G\to 1,\]
together with  an $E$-torsor structure on  the composite $Z \to X$, such that the following holds.
\begin{itemize}
    \item The action of $E$ on $Z$ extends that of $H$.
    \item The $G$-torsors $Z/H \to X$  and $Y \to X$ are isomorphic.
\end{itemize} 
\end{ques}

\noindent Of course, one does not expect to get a positive answer to this question in full generality. The goal of the article is to give both positive and negative results, depending on the nature of the groups $G$ and $H$.

Particular cases of Question \ref{ques} have been considered in \cite{HS}, \cite{BD}, \cite{McFaddin et cie} and \cite{ILA}, as well as \cite{BrionHB}, \cite{BrionPHB} and \cite{BrionHB3}. In \cite{HS}, \cite{BD} and \cite{ILA}, compositions of torsors are used to study obstructions to the local-global principle and to weak approximation over various arithmetically interesting fields, while in \cite{McFaddin et cie} they are used to study invariants of reductive groups. In \cite{BrionHB}, \cite{BrionPHB} and \cite{BrionHB3}, vector bundles (usual and projective) over abelian varieties, which are essentially compositions of torsors, are studied as interesting geometrical objects in their own right. \\

In the present article, we study compositions of torsors in a systematic way, at least in the case where $K$ has zero characteristic. The main positive result in this direction goes as follows.

\begin{thm}\label{thm principal}
Let $K$ be a field of characteristic $0$. Let $X$ be a connected smooth $K$-scheme. Let $G,H$ be smooth connected $K$-group schemes. Let $Y\to X$ be a $G$-torsor and let $Z\to Y$ be an $H$-torsor. Assume one of the following:
\begin{itemize}
    \item $H$ is an abelian variety.
    \item $H$ is a semi-abelian variety and $G$ is linear.
\end{itemize}
Then there exists a canonical extension of $K$-group schemes
\[\quad 1\to H\to E\to G\to 1,\]
together with a canonical structure of an $E$-torsor on the composite $Z\to X$ such that the following holds.
\begin{itemize}
    \item The action of $E$ on $Z$ extends that of $H$.
    \item The $G$-torsors $Z/H \to X$  and $Y\to X$ are isomorphic.
\end{itemize} 
\end{thm}

In order to prove this result, we first give in Section \ref{sec abstract results} an abstract statement (Theorem \ref{thm empilement abstrait 1}) for torsors and groups satisfying certain technical conditions. In Section \ref{sec proof}, we prove that these conditions are met in the cases given in Theorem \ref{thm principal}. In Section \ref{sec char pos}, we present a weaker version of Theorem \ref{thm principal} that works over arbitrary fields (cf.~Theorem \ref{thm principal car positive}).\\

Theorem \ref{thm principal} covers a certain number of the previously known results in the literature: \cite[Lem.~2.13]{BD} deals with the case where $H=\mathbb{G}_{\mathrm{m}}$ and $G$ is linear, \cite[Thm.~A.1.5]{McFaddin et cie} deals with the case where $X=\mathrm{Spec}(K)$, $H$ is a special torus and $G$ is reductive, while \cite[Thm.~A.1]{ILA} deals with the case where $H$ is a torus and $\mathrm{Pic}(\bar{G})=0$. In \cite{BrionHB}, \cite{BrionPHB} and \cite{BrionHB3}, Brion studies homogeneous bundles over abelian varieties, getting results that are related to our main theorem in the case where $X=\spec(K)$ and $G$ is an abelian variety, although they do not deal directly with compositions of torsors with such $G$ (see however \cite[Cor.~3.2]{BrionHB} and compare with our Theorem \ref{thm empilement abstrait 1} and Proposition \ref{prop inv necessaire}).

Theorem \ref{thm principal} does not cover all cases dealt with by Harari and Skorobogatov in \cite[Prop.~1.4]{HS}, since they consider $H$ to be of multiplicative type, and hence it may be non-connected. However, using a variant of the abstract Theorem \ref{thm empilement abstrait 1} (cf.~Theorem \ref{thm empilement abstrait 2}), we recover their result. Since they also provide an abstract result in their article (cf.~\cite[Thm.~1.2]{HS}), we compare this result with ours at the end of Section \ref{sec abstract results} (cf.~Remark \ref{rem HS}).\\

Finally, in Section \ref{sec counterexamples}, we present a certain number of counterexamples to Question \ref{ques}. Table \ref{table} summarizes both the positive and negative results we obtain in characteristic zero.\\

\begin{table}[h!]
\begin{center}
    \begin{tabular}{|c||c|c|c|c|}
    \hline
    \diagbox{$H$}{$G$} & t. & u. & s.s. & a.v. \\ \hline\hline
    t. & \checkmark & \checkmark & \checkmark & \xmark \\ \hline
    u.  & \xmark & \xmark & \xmark & \xmark \\ \hline
    s.s. & \xmark & \xmark & \xmark & \xmark \\ \hline
    a.v. & \checkmark & \checkmark & \checkmark & \checkmark \\ \hline
    \end{tabular}
    \hspace{3em}
    \framebox{\begin{tabular}{r l}
    t.~:& torus \\
    u.~:& unipotent \\
    s.s.~: & semisimple \\
    a.v.~: & abelian variety \\
    \checkmark~: & positive answer \\
    \xmark~: & negative answer
    \end{tabular}}
    \caption{Answer to Question \ref{ques} for several types of groups $G$ and $H$ over a field of characteristic zero.}\label{table}
\end{center}
\end{table}

\paragraph*{Acknowledgements.} The authors would like to warmly thank Michel Brion for his comments and suggestions, as well as two anonymous referees whose comments and questions helped us to improve the article. They also thank Ziyang Zhang, for his reading and comments. The third author's research was partially supported by ANID via FONDECYT Grant 1210010.

\section{Abstract results}\label{sec abstract results}

In this section, unless otherwise stated, $K$ is an arbitrary field. For a $K$-scheme $W$, we denote by $X_W,Y_W,Z_W,H_W,G_W$ the $W$-schemes obtained by base change from $X,Y,Z,H,G$ respectively. We start by proving the following abstract theorem, which will be the key tool to settle Theorem \ref{thm principal}.

\begin{thm}\label{thm empilement abstrait 1}
Let $K$ be a field. Let $X$ be a smooth $K$-scheme. Let $G,H$ be smooth connected $K$-group schemes with $H$ abelian. Let $Y\to X$ be a $G$-torsor and let $Z\to Y$ be an $H$-torsor. Finally, let $\cal M$ be the sheaf over the small smooth site over $K$ associated to the presheaf given by $(W\mapsto H(Y_W)/H(W))$. Assume the following:
\begin{itemize}
\item[(i)] The class of $Z_{\Omega}\to Y_{\Omega}$ in $H^1(Y_{\Omega},H_{\Omega})$ is $G(\Omega)$-invariant for every separably closed field $\Omega/K$.
\item[(ii)] The sheaf $\cal M$ is \'etale-locally isomorphic to the constant sheaf $\Z^n$ for a certain $n\in\N$. In particular, it is representable by a $K$-group-scheme $M$.
\end{itemize}
Then there exists a canonical extension of $K$-group schemes
\[\quad 1\to H\to E\to G\to 1,\]
together with a canonical structure of an $E$-torsor on the composite $Z\to X$ such that the following holds.
\begin{itemize}
    \item The action of $E$ on $Z$ extends that of $H$.
    \item The $G$-torsors $Z/H \to X$  and $Y\to X$ are isomorphic.
\end{itemize} 
\end{thm}

We start with a technical lemma, which might have an interest of its own.

\begin{lem}\label{lem surjectivite}
Let $K$ be a field and let $G$ be a connected (resp.~smooth connected) $K$-group scheme. Denote by $K(G)/K$ the function field of $G/K$, and by $\Omega$ an algebraic (resp.~separable) closure of $K(G)$. Let $\cal E$ be a contravariant group functor over the fppf (resp.~small smooth) site of $K$, equipped with a $K$-homomorphism $\pi:\cal E\to G$. Assume the following:
\begin{enumerate}[(a)]
\item The functor $\mathcal E$, on $K$-algebras, commutes with filtered direct limits.
\item The arrow $\pi(\Omega):\mathcal E(\Omega) \to G(\Omega)$ is surjective.
\end{enumerate}
Then, the following hold.
\begin{enumerate}[(1)]
    \item The arrow $\cal E(\bar K)\to G(\bar K)$ is surjective, where $\bar K$ is the algebraic (resp. separable) closure of $K$.
    \item There exists a finite set $I$ and an fppf (resp.~smooth) cover $(V_i \to G)_{i\in I}$ such that, for each $i\in I$, the arrow $V_i\to G$, considered as an element of $G(V_i)$, lifts via $\pi(V_i):\mathcal E(V_i) \to G(V_i)$. As a consequence, the arrow $\pi:\cal E\to G$ is surjective.
\end{enumerate}
\end{lem}

\begin{proof}
We prove (1). To do so, we may change the base field from $K$ to $\bar K$, reducing us to the case $K=\bar K$. Pick a point $g \in G(K) \subset G(\Omega)$. Let $e \in \mathcal E(\Omega)$ be a lift of $g$.  Let $U_0=\spec(A_0)$ be a non-empty affine open subscheme of $G$. Write $\Omega$ as the direct limit (union) $\varinjlim A_j$, of its flat (resp.~smooth) and finitely presented $A_0$-subalgebras $A_j$. By condition (a), $e$ belongs to $\mathcal E(A_j)$ for some $j$. Since fppf (resp.~smooth) morphisms are open (cf. \cite[Tag~01UA]{SP}), in geometric terms, there exists a non-empty affine open $U \subset U_0 \subset G$, and an fppf (resp.~smooth) morphism $V:=\spec(A_j) \to U$, such that $e$ belongs to $\cal E(V)$. Since $V$ is a non-empty (resp.~smooth) $K$-scheme of finite-type and $K= \bar K$, there exists a closed point $v: \spec(K) \to V$. Then, the image of $e$ via the morphism $\cal E(V) \to \cal E(K)$ induced by $v$ is the desired lift of $g$.\\

We prove (2). Denote by $g \in G(\Omega)$ the generic point of $G$. Let $e \in \mathcal E(\Omega)$ be a lift of $g$, which exists by condition (b). The same limit argument used in (1), produces an affine open $U \subset G$, and an fppf (resp.~smooth) morphism $V\to U$, such that the composite $V \to U \to G$ lifts via $\pi$, to an element of $\mathcal E(V)$. Since $G(\bar K)$ is Zariski-dense in $G$, the translates $\gamma\cdot V$, for $\gamma \in G(\bar K)$, form an fppf (resp.~smooth) cover of $G_{\bar K}$, from which we may extract a finite cover. Since these $\gamma$'s lift to $\mathcal E(\bar K)$ by (1), there exists a finite (resp. finite separable) field extension $L/K$, such that a cover of $G_L$ exists, with the required property. Composing with the projection $G_L \to G$, which is finite (resp. finite separable) and locally free, hence fppf (resp.~smooth), gives such a cover over $K$.

For the last assertion, consider an arbitrary morphism of $K$-schemes (resp.~smooth $K$-schemes) $\phi:S\to G$ and define the fppf (resp.~smooth) cover $(S\times_G V_i\to S)_{i\in I}$ of $S$ by pullback. Since the map $\phi_i:S\times_G V_i\to G$ induced by $\phi$ factors through $V_i\to G$, which lifts to $\cal E(V_i)$, we see that $\phi_i$ lifts to $\cal E(S\times_G V_i)$ and the surjectivity follows.
\end{proof}

The following is an easy exercise given the actual literature, but we state it here since it is used several times in what follows.

\begin{lem}\label{lem representabilite}
Let $K$ be a field and let $G,H$ be $K$-group schemes. Let $G^0,H^0$ denote the corresponding neutral connected components. Assume that $G^0,H^0$ are smooth and that $F:=H/H^0$ is \'etale-locally isomorphic to $\Z^n$ for some $n\in\N$. Let
\[1\to H\to \cal E\to G\to 1,\]
be an exact sequence of sheaves over the fppf (resp.~small smooth) site of $K$. Then $\cal E$ is representable by a $K$-group scheme $E$.
\end{lem}

\begin{proof}
Since $H$ is normal in $\cal E$ and $H^0$ is characteristic in $H$, we obtain that $H^0$ is normal in $\cal E$. We may then quotient $\cal E$ by $H^0$ in order to get an exact sequence
\[1\to F\to\cal E'\to G\to 1.\]
Since $F$ is locally isomorphic to $\Z^n$ we know by \cite[Exp.~8, Prop.~5.1]{SGA7} that every $F$-torsor over any connected component of $G$ is locally trivial, hence representable. In particular, this is the case for $\cal E'$. Thus we have an exact sequence of $K$-group schemes
\[1\to F\to E'\to G\to 1,\]
which gives us the following exact sequence of sheaves
\begin{equation}\label{eqn suite H0}
1\to H^0\to \cal E\to E'\to 1.
\end{equation}
Now, by Chevalley's Theorem (cf.~\cite[Thm.~1.1]{ConradChevalley} or \cite[\S9.2, Thm.~1]{NeronModels}), there is an exact sequence
\[1\to L\to H^0\to A\to 1,\]
with $L$ linear and $A$ an abelian variety. Since $L$ is a characteristic subgroup of $H^0$ and $H^0$ is normal in $\cal E$, we see that $L$ is normal in $\cal E$. We may then quotient $\cal E$ by $L$ in order to get an exact sequence
\[ 1\to A \to \cal E''\to E'\to 1.\]
Then $\cal E''$ is representable by \cite[III, Thm.~4.3.(c)]{Milne} since $A$ is smooth, proper and connected over $K$. Thus we have an exact sequence of $K$-group schemes
\[1\to A\to E''\to E'\to 1,\]
which gives us the following exact sequence of sheaves
\[1\to L\to \cal E\to E''\to 1.\]
Then $\cal E$ is representable by \cite[III, Thm.~4.3.(a)]{Milne} since $L$ is affine.

Finally, note that Milne's results \cite[III, Thm.~4.3.(a),(c)]{Milne} are stated over the fppf site. However, if we are over the small smooth site, since $H^0$ is smooth and $\cal E$ is an $H^0$-torsor (cf.~sequence \eqref{eqn suite H0}), we deduce that it corresponds to a unique $H^0$-torsor over the fppf site (cf.~\cite[Thm.~11.7.1, Rem.~11.8.3]{GrothendieckBrauerIII}), which is then representable. This implies the representability of $\cal E$ over the small smooth site (by the same $K$-scheme).
\end{proof}

We are now ready to prove Theorem \ref{thm empilement abstrait 1}.

\begin{proof}[Proof of Theorem \ref{thm empilement abstrait 1}]
For a $K$-scheme $W$, consider the group $\aut_{X_W}^{H}(Z_W)$ of $X_W$-auto\-morphisms $\varphi$ of $Z_W$ that are compatible with the action of $H$ in the sense that the following diagram commutes:
\[\xymatrix{
H\times Z_W \ar[r]^-{a_W} \ar[d]_{\id\times\varphi} & Z_W \ar[d]^{\varphi} \\
H\times Z_W \ar[r]^-{a_W} & Z_W,
}\]
where $a$ denotes the morphism defining the action of $H$ on $Z$ and $a_W$ the corresponding morphism after base change. The functor $W\mapsto \aut_{X_W}^{H}(Z_W)$ defines a group presheaf over the small smooth site over $K$. Denote by $\underline{\aut}_X^{H}(Z)$ the corresponding sheaf and consider the subsheaf $\underline{\aut}_{Y}^{H}(Z)$ defined by taking the subgroup $\aut_{Y_W}^{H}(Z_W)$ of $\aut_{X_W}^{H}(Z_W)$ for each $W$. Since every element in $\aut_{X_W}^{H}(Z_W)$ induces an $X_W$-automorphism of $Y_W$, we have an exact sequence of sheaves
\[1\to \underline{\aut}_{Y}^{H}(Z)\to \underline{\aut}_X^{H}(Z)\xrightarrow{\pi} \underline{\aut}_X(Y),\]
where $\underline{\aut}_X(Y)$ denotes the sheaf of $X$-automorphisms of $Y$.

Note that $G$ is naturally a subgroup of $\underline{\aut}_X(Y)$.  Taking the pullback via $\pi$, we get an exact sequence of sheaves
\[1\to \cal A\to \cal E'\xrightarrow{\pi} G,\]
where $\cal A=\underline{\aut}_{Y}^{H}(Z)$. Now, the functor $V/Y\mapsto \aut_{V}^{H}(Z\times_Y V)$ over the small smooth site of $Y$ is represented, as a $Y$-scheme, by $H_Y$ (cf.~for instance \cite[III.\S1.5]{Giraud}). In other words, $\cal A(W)=\aut_{Y_W}^{H}(Z_W)=H(Y_W)$ and hence $\cal M=\cal A/H$. Thus, by (ii), we get an exact sequence of group sheaves over $K$
\[1\to H\to \cal A\to M\to 1,\]
where $M$ is a $K$-group scheme that is \'etale-locally isomorphic to $\Z^n$ for some $n\in\N$. By Lemma \ref{lem representabilite}, it follows that $\cal A$ is represented  by a $K$-group scheme $A$. Moreover, since $H$ is connected, it corresponds to the neutral connected component of $A$. In particular, $H$ is a characteristic subgroup of $A$.

On the other hand, since $X$, $Z$, $H$ and $G$ are of finite type over $K$, the functor $\cal E' \subset \underline{\aut}_{K}(Z)$ commutes with direct limits. Moreover, by (i), if we set $\Omega:=\overline{K(G)}$, we know that $Z_\Omega$ is isomorphic to $g^*Z_\Omega$ as an $H$-torsor over $G_\Omega$ for every $g\in G(\Omega)$, and thus the arrow $\cal E'(\Omega)\to G(\Omega)$ is surjective. Then, by Lemma \ref{lem surjectivite}, we get that the arrow $\pi:\cal E'\to G$ is surjective. In particular, we get an exact sequence of sheaves
\[1\to A \to \cal E'\xrightarrow{\pi} G\to 1.\]
And since $M=A/H$ is locally isomorphic to $\Z^n$ by (ii), we see by Lemma \ref{lem representabilite} that $\cal E'$ is representable. Thus we have an exact sequence of $K$-group \emph{schemes}
\begin{equation}\label{eqn seq A & G}
 1\to A \to E'\xrightarrow{\pi} G\to 1.\tag{$\mathcal S$}
\end{equation}
Since $A$ is normal in $E'$ and $H$ is characteristic in $A$, we obtain that $H$ is normal in $E'$. We may then quotient by $H$ in order to get an exact sequence
\begin{equation}\label{eqn seq M & G}
1\to M\to F\xrightarrow{\bar{\pi}} G\to 1.\tag{$\bar{\mathcal S}$}
\end{equation}
Since $M$ is discrete and torsion-free by (ii), and since $G$ is connected, we see that the neutral connected component $F^0 \subset F$, is mapped isomorphically to $G$ by $\bar{\pi}$. This provides a \emph{canonical} splitting of extension \eqref{eqn seq M & G}. As a consequence, extension \eqref{eqn seq A & G} is the pushout of an extension of group schemes (obtained as a pullback via the canonical splitting)
\[1\to H\to E\to G\to 1.\]
As a subgroup of $\underline{\aut}_X^H(Z)$, it acts on $Z$, and it is immediate to check then that $Z\to X$ is an $E$-torsor, which enjoys the required properties. To conclude, note that the  construction above is canonical. 
\end{proof}

In the previous theorem, the assumptions that $H$ is abelian and that $G$ and $H$ are both connected can be removed when $M$ is the trivial group. In that way, one gets the following result, which implies \cite[Prop.~1.4]{HS}.

\begin{thm}\label{thm empilement abstrait 2}
Let $K$ be a field. Let $X$ be a smooth $K$-scheme. Let $G,H$ be smooth $K$-group schemes. Let $Y\to X$ be a $G$-torsor and let $Z\to Y$ be an $H$-torsor. Assume the following:
\begin{itemize}
    \item[(i)] The class of $Z_{\Omega}\to Y_{\Omega}$ in $H^1(Y_{\Omega},H_{\Omega})$ is $G(\Omega)$-invariant for every separably closed field $\Omega/K$.
    \item[(ii')] The sheaf of sets $\cal M$ over the small smooth site over $K$ associated to the presheaf given by $(W\mapsto H(Y_W)/H(W))$ is trivial.
\end{itemize}
Then there exists a canonical extension of $K$-group schemes
\[\quad 1\to H\to E\to G\to 1,\]
together with a canonical structure of an $E$-torsor on the composite $Z\to X$ such that the following holds.
\begin{itemize}
    \item The action of $E$ on $Z$ extends that of $H$.
    \item The $G$-torsors $Z/H \to X$  and $Y\to X$ are isomorphic.
\end{itemize} 
\end{thm}

\begin{proof}
The proof starts exactly as the one above, except for the following modification. Instead of considering the groups $\aut_{X_W}^{H}(Z_W)$ and $\aut_{Y_W}^{H}(Z_W)$ for smooth $W\to \mathrm{Spec}(K)$ and the corresponding sheaves $\underline{\aut}_X^{H}(Z)$ and $\underline{\aut}_Y^{H}(Z)$, we consider the groups $\aut_{X_W}^{H'}(Z_W)$ and $\aut_{Y_W}^{H'}(Z_W)$ and the corresponding sheaves $\underline{\aut}_X^{H'}(Z)$ and $\underline{\aut}_Y^{H'}(Z)$, where $H'$ is the $Y$-group scheme obtained by twisting $H_Y$ by the torsor $Z\to Y$ ($H'$ is actually $H_Y$ when $H$ is abelian). This group scheme acts naturally on $Z$ on the left compatibly with the right action of $H$ (cf.~\cite[III.\S1.5]{Giraud}). In particular, we still have the equality $\aut_{Y_W}^{H'}(Z_W)=H(Y_W)$ by \textit{loc.~cit.} and an exact sequence
\[1\to \cal A'\to \cal E\xrightarrow{\pi} G\to 1,\]
with $\cal A'=\underline{\aut}_{Y}^{H'}(Z)$, where the surjectivity of $\pi$ is given once again by Lemma \ref{lem surjectivite}. The assumption (ii') on the sheaf $\cal M$ tells us then that $\cal A'$ is actually $H$, and hence the exact sequence becomes
\[1\to H\to \cal E\to G\to 1.\]
Thus $\cal E$ is an $H$-torsor, which is then representable by a $K$-group scheme $E$ by Lemma \ref{lem representabilite}. And again, since $E$ is by definition a subgroup of $\underline{\aut}_X^H(Z)$, it is immediate to check that $E$ acts on $Z$ and that $Z\to X$ is an $E$-torsor. The fact that the construction is canonical is once again easy to see.
\end{proof}

\begin{rem}
As a referee pointed out, the proofs of Theorems \ref{thm empilement abstrait 1} and \ref{thm empilement abstrait 2} do not use the fact that $Y\to X$ is a $G$-torsor, but rather that $G$ acts on the $X$-scheme $Y$. One may extend thus the statements of both theorems to a more general setting (for instance, one can consider projective bundles or Severi-Brauer schemes over $X$, which have natural actions by forms of $\mathrm{PGL}_n$). However, we were not able to come up with new applications in this setting.
\end{rem}

In Theorems \ref{thm empilement abstrait 1} and \ref{thm empilement abstrait 2}, the $G(\Omega)$-invariance of the $H_{\Omega}$-torsor $Z_{\Omega}\to Y_{\Omega}$ is, in a wide variety of cases, a strictly necessary hypothesis in order to get a positive answer to Question \ref{ques}. More precisely:

\begin{prop}\label{prop inv necessaire}
Let $K$ be an algebraically closed field of characteristic $0$. Let
\[1\to H\to E\to G\to 1,\]
be an extension of smooth $K$-group schemes with $G$ connected. Assume that the unipotent radical of $H$ is trivial. Let $X$ be a smooth $K$-scheme, let $Z\to X$ be an $E$-torsor and let $Y:=Z/H$, so that $Z\to Y$ is an $H$-torsor and $Y\to X$ is a $G$-torsor. Then the class of $Z\to Y$ in $H^1(Y,H)$ is $G(K)$-invariant.
\end{prop}

\begin{proof}
Define $C$ to be the centralizer of $H$ in $E$. We claim that $C$ surjects onto $G$ via the projection. Since $C$ is the kernel of the natural arrow $E(K)\to\aut(H)$ given by conjugation, the claim amounts to proving that the induced morphism $G(K)\to\out(H)$ is trivial, where $\out(H):=\aut(H)/\inn(H)$.

By Lemma \ref{spandiv}, which we prove in the following section and uses the hypothesis on the characteristic of $K$, we know that $G(K)$ is generated by its infinitely divisible elements, while $\out(H)$ has no such elements. Indeed, this group is finite for reductive groups (cf.~\cite[Thm.~5.2.3]{Demazure}), while it is a subgroup of $\gln(\Z)$ for abelian varieties (as follows from ~\cite[Thm.~10.15]{MilneAV}). In the general case, our hypothesis on $H$ and Chevalley's Theorem (cf.~\cite[Thm.~1.1]{ConradChevalley}) ensure that $H$ is an extension
\[1\to L\to H\to A\to 1,\]
of an abelian variety $A$ and a reductive linear group $L$. Since $L$ is a characteristic subgroup of $H$ and scheme morphisms from an abelian variety to a linear group are constant, one easily sees that $\aut(H)$ is isomorphic to a subgroup of $\aut(L)\times\aut(A)$. We deduce the same property for $\out(H)$, which implies the claim.

Now consider $g\in G(K)$ and let us prove that the torsor $g^*Z\to Y$, defined as the left vertical arrow of the fiber product
\[\xymatrix{
g^*Z \ar[r] \ar[d] & Z \ar[d] \\
Y \ar[r]^{g} & Y,
}\]
is isomorphic to the torsor $Z\to Y$. Let $c\in C(K)\subset E(K)$ be a preimage of $g$. Then we have a commutative square
\[\xymatrix{
Z \ar[d] \ar[r]^c & Z \ar[d] \\
Y \ar[r]^g & Y.
}\]
Then, by the universal property of the fiber product, we get a $Y$-morphism $Z\to g^*Z$, which we claim it is $H$-equivariant. This is a straightforward computation that uses the fact that $c\in C(K)$ commutes with $H$. This proves that the class of $Z\to Y$ is $g$-invariant and hence $G(K)$-invariant.
\end{proof}

\begin{rem}\label{rem HS}
A result similar to Theorem \ref{thm empilement abstrait 2} can be found in \cite[Thm.~1.2]{HS}. However, the assumptions are slightly different:

Harari and Skorobogatov assume that every morphism $Z_{\bar K}\to H_{\bar K}$ is trivial. This is easily seen to imply the triviality of $\cal M$ and hence our assumption (ii').

On the other hand, they assume that every automorphism of $Y_{\bar K}$ given by an element $g\in G(\bar K)$ can be lifted to an automorphism of $Z_{\bar K}$. Our assumption (i) implies this, of course, but it is not clear whether they are equivalent assumptions, even though ours seems to be always necessary, as it can be seen from Proposition \ref{prop inv necessaire}.

In any case, assumption (ii') of Theorem \ref{thm empilement abstrait 2} is met for instance when $H$ is affine, $G$ is anti-affine and $X$ is connected and proper. Indeed, in this case $\cal O(Y)=K$ and hence $H(Y\times W)=H(W)$ for geometrically integral $W$ by \cite[Lem.~5.2]{BrionHom}. This implies the triviality of $\cal M$. These are milder hypotheses than those considered by Harari and Skorobogatov in \cite[Prop.~1.4]{HS}, who deal for instance with the case of $H$ of multiplicative type and $Y$ proper.
\end{rem}

\section{Proof of Theorem \ref{thm principal}}\label{sec proof}

It will suffice to prove that the assumptions (i) and (ii) of Theorem \ref{thm empilement abstrait 1} are met under each of the hypotheses of Theorem \ref{thm principal}. We fix then a field $K$ of characteristic $0$ and keep the other notations as above: $X$ is a connected smooth $K$-scheme; $G,H$ are smooth connected $K$-group schemes; $Y\to X$ is a $G$-torsor and $Z\to Y$ is an $H$-torsor; $\cal M$ is the sheaf over the small smooth site over $K$ associated to the presheaf given by $(W\mapsto H(Y_W)/H(W))$.\\

We prove (ii) first. By \'etale descent, we may assume that $K$ is algebraically closed and we need to prove that $\cal M$ is representable and isomorphic to $\Z^n$. This is a direct consequence of a result of Rosenlicht, which we prove in the appendix in the context of separably closed fields (cf.~Lemma \ref{RosenLem}).

We are then left with the proof of (i), which is clearly implied by the following result.

\begin{prop}\label{prop action triviale}
Let $K$ be an algebraically closed field of characteristic $0$. Let $G$ and $H$ be smooth connected $K$-algebraic groups and make one of the following assumptions:
\begin{itemize}
    \item[(a)] $H$ is an abelian variety.
    \item[(b)] $G$ is linear and $H$ is a semi-abelian variety.
\end{itemize}
Let $X$ be a smooth $K$-scheme and let $Y\to X$ be a $G$-torsor. Then the action of $G(K)$ on $H^1(Y,H)$ is trivial.
\end{prop}

\begin{proof}[Proof of Proposition \ref{prop action triviale}.(a)]
According to \cite[Cor.~XIII.2.4, Prop.~XIII.2.6]{raynaud}, the group $H^1(Y,H)$ is torsion. Hence, given an element $a\in H^1(Y,H)$, it comes from $H^1(Y,H[n])$ for some $n>0$. By \cite[Thm.~5.2]{BD}, we deduce that $a$ is fixed by $G(K)$.
\end{proof}

We prove now Proposition \ref{prop action triviale}.(b). For that purpose, we first need to prove some lemmas on the structure of the groups involved. In all of them, we keep the notations of Proposition \ref{prop action triviale}.

\begin{lem}\label{tcof}
Let $A$ be an abelian variety over $K$. Then the group $H^1(Y,A)$ is torsion of cofinite type, i.e. its $m$-torsion subgroup is finite for every $m\in\N$.
\end{lem}

\begin{proof}
As it was already noted in the proof of Proposition \ref{prop action triviale}.(a), the group $H^1(Y,A)$ is torsion. Moreover, by \cite[Exp.~XVI, Thm.~5.2]{SGA4}, for each integer $n>0$, the group $H^1(Y,A[n])$ is finite, and hence so is its quotient $H^1(Y,A)[n]$.
\end{proof}

\begin{lem}\label{spandiv}
The group $G(K)$ is spanned by its divisible subgroups.
\end{lem}

\begin{proof}
Write $G=G_{\mathrm{aff}}G_{\mathrm{ant}}$ where $G_{\mathrm{aff}}$ is the largest connected affine subgroup of $G$ and $G_{\mathrm{ant}}$ is the largest anti-affine subgroup of $G$ (cf.~\cite[Thm.~1.2.4]{BrionSamuelUma}). Every element $g$ of $G_{\mathrm{aff}}(K)$ can be written as:
$$g=su_1...u_r$$
where $s$ is a semisimple element of $G_{\mathrm{aff}}(K)$ and each $u_i$ is contained in a subgroup of $G_{\mathrm{aff}}$ isomorphic to $\mathbb{G}_\mathrm{a}$. Hence $G_{\mathrm{aff}}(K)$ is spanned by its divisible subgroups. Moreover, the anti-affine group $G_{\mathrm{ant}}$ is connected commutative (cf.~\cite[Thm.~1.2.1]{BrionSamuelUma}), and hence $G_{\mathrm{ant}}(K)$ is a divisible group. We deduce that $G(K)$ is spanned by its divisible subgroups.
\end{proof}

\begin{lem}\label{nodiv}
 Let $\Gamma$ be a profinite group. Then, $\Gamma$ has no non-trivial infinitely divisible elements.
\end{lem}

\begin{proof}
The statement is clear if $\Gamma$ is finite. It thus also holds for inverse limits of finite groups.
\end{proof}

\begin{proof}[Proof of Proposition \ref{prop action triviale}.(b)]
We have an exact sequence:
$$0 \rightarrow T \rightarrow H \rightarrow A\rightarrow 0,$$
where $T$ is a torus and $A$ is an abelian variety. It induces a cohomology exact sequence:
$$H^1(Y,T) \xrightarrow[]{f} H^1(Y,H) \xrightarrow[]{g} H^1(Y,A),$$
whose arrows are clearly $G(K)$-equivariant since the action is on $Y$. Put $M:=\mathrm{im}(g)$ and $N:=\mathrm{im}(f)$, so that we have the exact sequence of $G(K)$-groups
\[0\to N\to H^1(Y,H)\to M\to 0.\]
By Lemma \ref{tcof}, $H^1(Y,A)$ is torsion of cofinite type, and hence so is $M$. Moreover, by Proposition \ref{prop action triviale}.(a), the group $G(K)$ acts trivially on $H^1(Y,A)$, and hence on $M$. On the other hand, note that $H^1(Y,T)\cong \mathrm{Pic}(Y)^{\dim(T)}$. Since $G$ is linear, a result of Sumihiro (cf.~\cite[Thm.~5.2.1]{BrionLin}) tells us that the action of $G(K)$ on $H^1(Y,T)$ is trivial, hence also its action on $N$.

Thus, the action of $G(K)$ on $H^1(Y,H)$ corresponds to a morphism from $G(K)$ to $\mathrm{Hom}(M,N)$. The abelian group $M$ is torsion, and hence $\mathrm{Hom}(M,N)=\mathrm{Hom}(M,N_{\mathrm{tors}})$. Moreover, $M$ is of cofinite type, and so is the group $N_{\mathrm{tors}}$ since it is isomorphic to a quotient of $\mathrm{Pic}(Y)^{\dim(T)}$. We can therefore write:
\begin{gather*}
    M \cong \bigoplus_p \left(F_p \oplus (\mathbb{Q}_p/\mathbb{Z}_p)^{r_p}\right),\\
    N_{\mathrm{tors}} \cong \bigoplus_p \left(F'_p \oplus (\mathbb{Q}_p/\mathbb{Z}_p)^{r'_p}\right),
\end{gather*}
where $p$ runs through the set of all prime numbers, $F_p$ and $F_p'$ are finite abelian $p$-groups and $r_p,r'_p\geq 0$. It follows that $\mathrm{Hom}(M,N_{\mathrm{tors}})$ is a profinite group. Hence, it has no non-trivial infinitely divisible elements by Lemma \ref{nodiv}. Thus, every morphism from $G(K)$ to $\mathrm{Hom}(M,N_{\mathrm{tors}})$ is trivial by Lemma \ref{spandiv}. We deduce that the action of $G(K)$ on $H^1(Y,H)$ is trivial.
\end{proof}

We finish this section with an example that shows that one really needs to assume $K$ to be algebraically closed in Proposition \ref{prop action triviale}.(b) (and hence separably closed in assumption (i) of Theorems \ref{thm empilement abstrait 1} and \ref{thm empilement abstrait 2}).

\begin{exa}
Let $K$ be a field and let $L/K$ be a separable quadratic field extension such that the norm $N_{L/K}: L^\times \to K^\times $ is not surjective. Consider the extension of algebraic $K$-groups
\[1 \to R^1_{L/K}(\gm) \to R_{L/K}(\gm) \xrightarrow{N_{L/K}}\gm \to 1,\]
where $R_{L/K}$ denotes Weil scalar restriction and $N_{L/K}$ is the norm of $L/K$. Set $G=Y:=\gm$, $H:=R^1_{L/K}(\gm)$ and $X:=\spec(K)$. The extension above provides a class
\[x_0:=[R_{L/K}(\gm) \to \gm] \in H^1(Y,H).\]
This class is not invariant under the action of $G(K)$. Indeed, the action of $G(K)$ is described as follows:
\[\lambda\cdot x=x+p^*\delta(\lambda),\]
where $\lambda \in G(K)$, $x\in H^1(Y,H)$, $p:Y\to \mathrm{Spec}(K)$ is the structure morphism and $\delta:G(K)\to H^1(K,H)$ is the connecting map in Galois cohomology. In particular, since the arrow
\[p^*:H^1(K,H)\to H^1(Y,H),\]
is injective, we have $\lambda\cdot x=x$ if and only if $\lambda \in N_{L/K}(L^\times)$, which does not hold in general.
\end{exa}

\section{Positive characteristic}\label{sec char pos}

In this section, we present a weaker version of Theorem \ref{thm principal} that works in positive characteristic.

\begin{thm}\label{thm principal car positive}
Let $K$ be a field. Let $X$ be a connected smooth $K$-scheme. Let $G,H$ be smooth connected $K$-group schemes. Let $Y\to X$ be a $G$-torsor and let $Z\to Y$ be an $H$-torsor. Assume one of the following:
\begin{itemize}
    \item $H,G$ are abelian varieties and $X$ is proper.
    \item $H$ is a torus and $G$ is linear.
\end{itemize}
Then there exists a canonical extension of $K$-group schemes
\[\quad 1\to H\to E\to G\to 1,\]
together with a canonical structure of an $E$-torsor on the composite $Z\to X$.
\end{thm}

\begin{proof}
The proof of this result is given once again by Theorem \ref{thm empilement abstrait 1}, which is valid over an arbitrary field. We need to prove then that assumptions (i) and (ii) of Theorem \ref{thm empilement abstrait 1} are met. The proof of (ii) is exactly as before: By \'etale descent, we may assume that $K$ is separably closed and we need to prove that $\cal M$ is representable and isomorphic to $\Z^n$. This is a direct consequence of Lemma \ref{RosenLem}, which is valid over separably closed fields.

Thus, we are only left with (i). In the second case, this is a direct consequence of Sumihiro's result we used before (cf.~\cite[Thm.~5.2.1]{BrionLin}). In the first case, (i) is implied by Proposition \ref{prop action triviale H=A car pos} here below.
\end{proof}

\begin{prop}\label{prop action triviale H=A car pos}
Let $K$ be an separably closed field. Let $X$ be a smooth proper $K$-scheme. Let $G$ and $H$ be abelian varietes and let $Y\to X$ be a $G$-torsor. Then the action of $G(K)$ on $H^1(Y,H)$ is trivial.
\end{prop}

\begin{proof}
Without loss of generality, we can assume that $X$ (and hence $Y$) is connected. Moreover, since $H$ is smooth, we may and do assume that $K$ is algebraically closed.\\
Note that $H^1(Y,H)$ is torsion by \cite[Cor.~XIII.2.4, Prop.~XIII.2.6]{raynaud}. As in the proof of Proposition \ref{prop action triviale}, \cite[Thm.~5.2]{BD} implies that $G(K)$ acts trivially on the $q$-primary part of $H^1(Y,H)$ for every prime $q\neq p$. It is therefore enough to prove that $G(K)$ also acts trivially on $H^1(Y,H)[p^n]$ for every $n\in\N$. We proceed by induction on $n$.\\

For $n=1$, we know that $H^1(Y,H)[p]$ is a quotient of $H^1_{\mathrm{fppf}}(Y,H[p])$ and $H[p]\cong (\mathbb{Z}/p\mathbb{Z})^a\times (\mu_p)^b \times (\alpha_p)^c$ for some integers $a,b,c$ (cf.~\cite{Shatz86}). It is therefore enough to prove that the action of $G(K)$ on $H^1_{\mathrm{fppf}}(Y,\mathbb{Z}/p\mathbb{Z})$ and $H^1_{\mathrm{fppf}}(Y,\mu_p)$ and $H^1_{\mathrm{fppf}}(Y,\alpha_p)$ is trivial.

Since $Y$ is proper over $K$, \cite[VI, Cor.~2.8]{Milne} ensures the finiteness of $H^1_{\mathrm{fppf}}(Y,\mathbb{Z}/p\mathbb{Z})$, which implies the triviality of the action in this case.

Now by Kummer theory, we have an exact sequence:
$$0 \rightarrow K[Y]^{\times}/(K[Y]^{\times})^p \rightarrow H^1_{\mathrm{fppf}}(Y,\mu_p) \rightarrow \mathrm{Pic}(Y)[p] \rightarrow 0.$$
Using the properness of $Y$ once more, we have $K[Y]=K$, and hence the quotient $K[Y]^{\times}/(K[Y]^{\times})^p$ is trivial. Moreover, the group $\mathrm{Pic}(Y)[p]$ is always finite. Hence $H^1_{\mathrm{fppf}}(Y,\mu_p)$ is finite, which implies the triviality of the action in this case.

We deal now with $H^1_{\mathrm{fppf}}(Y,\alpha_p)$. Let $A$ be a $K$-algebra. Then  $H^0(Y_A,\cal O_{Y_A})=A$ and $H^1(Y_A,\cal O_{Y_A})=H^1_{\mathrm{fppf}}(Y,\cal O_{Y})\otimes_K A$, so that, after taking cohomology of  the extension of fppf sheaves (over $Y_A$) $$0 \to \alpha_p  \to \mathbb G_a \xrightarrow{\mathrm{Frob}}   \mathbb G_a \to 0,$$ we get  an exact sequence
$$ 0 \to A/A^p \to H^1_{\mathrm{fppf}}(Y_A,\alpha_p) \to H^1(Y,\cal O_{Y})\otimes_K A.$$   Taking $A=K$, we get an inclusion of finite-dimensional $K$-vector spaces $$ H^1_{\mathrm{fppf}}(Y,\alpha_p) \subset  H^1(Y,\cal O_{Y}).$$ It then suffices to show that $G(K)$ acts trivially on
$H^1_{\mathrm{fppf}}(Y,\cal O_{Y})$. To do so, observe that the $G$-action on $Y$, induces an action of the abstract group $G(A)$ on the $A$-scheme $Y_A$, and hence an $A$-linear action of $G(A)$ on $H^1_{\mathrm{fppf}}(Y,\cal O_{Y})\otimes_K A$. Being functorial in $A$, it arises from a morphism of algebraic $K$-groups $\rho: G \rightarrow \mathrm{GL}(H^1_{\mathrm{fppf}}(Y,\cal O_{Y})), $ which is trivial because $G$ is an abelian variety. This concludes the proof for $n=1$.\\

Consider now the following exact sequence
\[0\to H^1(Y,H)[p]\to H^1(Y,H)[p^{n+1}]\to H^1(Y,H)[p^n],\]
and let $I$ be the image of the rightmost arrow. By the inductive assumption, the group $G(K)$ acts trivially on $H^1(Y,H)[p]$ and on $I$. Hence the action of $G(K)$ on $H^1(Y,H)[p^{n+1}]$ corresponds to a morphism $G(K)\to \mathrm{Hom}(I, H^1(Y,H)[p])$. But this morphism is trivial since $G(K)$ is divisible and $\mathrm{Hom}(I, H^1(Y,H)[p])$ is $p$-torsion. We deduce that $G(K)$ acts trivially on $H^1(Y,H)[p^{n+1}]$, as wished.
\end{proof}

\section{Counterexamples}\label{sec counterexamples}

In this section, we provide examples of towers of torsors that do not admit a torsor structure under an extension of the two involved groups. We treat every negative case considered in Table \ref{table}.

\subsection{Examples where $H$ is a torus}
As it is suggested by Lemma \ref{RosenLem}, when $H$ is a torus, assumption (ii) of Theorem \ref{thm empilement abstrait 1} is satisfied in all generality. According to Table \ref{table}, it is then assumption (i), on the $G(\Omega)$-invariance of the $H$-torsor $Z\to Y$ that must fail in order to get counterexamples.

\begin{exa}
Assume that $K$ is algebraically closed of characteristic $0$, $X=\spec(K)$, $G=Y$ is an elliptic curve, and $H=\mathbb{G}_\mathrm{m}$. Then the group $G(K)$ acts on $H^1(Y,H)=\pic(G)$ via the following formula:
$$Q \cdot [D]= [D] + \deg(D) \cdot ([Q]-[O]),\quad Q\in G(K),\, [D]\in\pic(G).$$
This action is not trivial, and hence one can find a class in $\pic(G)$ that is not $G(K)$-invariant. By Proposition \ref{prop inv necessaire}, this class represents an $H$-torsor $Z\to Y$ such that the composition $Z\to \mathrm{Spec}(K)$ is not a torsor under an extension $E$ of $G$ by $H$.
\end{exa}

\subsection{Examples where $H$ is unipotent}

We continue with the case in which $H$ is unipotent. The following example covers the cases in which $G$ is either a torus, a unipotent group or a semisimple group.

\begin{exa}\label{ex G=aff H=Ga}
Let $H=\ga$, $X$ an elliptic curve over an algebraically closed field $K$ of characteristic zero and $Y$ the trivial $G$-torsor with $G$ either $\ga$, $\gm$ or $\sln$ (with $n\geq 2$).

On the one hand, by Künneth's formula we have $$H^1(Y,\mathbb{G}_a)=H^1(X,\ga)\otimes_K \cal O(G)=\cal O(G),$$ which is an infinite-dimensional $K$-vector space.

On the other hand, every extension $E$ of $G$ by $\ga$ is split by the basic theory of linear groups. Even more, if $G=\ga$ or $G=\sln$, then the extension is simply the direct product, while if $G=\gm$, then it corresponds to the semi-direct product $E_k:=\ga\rtimes_{k}\gm$ with $\gm$ acting on $\ga$ by a character of the form \begin{align*}
    \chi_k: \gm &\rightarrow \mathrm{Aut}(\ga)=\gm, \\x &\mapsto x^k,
\end{align*}
for some $k\in\Z$. In particular, these extensions are parametrized by $\Z$.

Thus, in the former two cases (where $G=\ga$ or $G=\sln$), we get that $E$-torsors lifting $Y\to X$ are classified by the one-dimensional vector space $H^1(X,\ga)$, while $H$-torsors $Z\rightarrow Y$ are classified by the infinite dimensional vector space $H^1(X,\ga)\otimes_K \cal O(G)$. This tells us that there exist $H$-torsors $Z\rightarrow Y$ such that the composite $Z\rightarrow Y\rightarrow X$ is not a torsor under an extension of $G$ by $H$.

In the latter case where $G=\gm$ and $E_k=\ga\rtimes_{k}\gm$, we have a split exact sequence:
\[1\to H^1(X,\ga) \to H^1(X,E_k) \to H^1(X,G)\to 1.\]
Since the torsor $Y \rightarrow X$ is trivial, we deduce that $E_k$-torsors lifting $Y\to X$ are classified by the one-dimensional vector space $H^1(X,\ga)$. Consider then the composite
\[H^1(X,\ga)\to H^1(X,E_k)\to H^1(Y,\ga),\]
where the first arrow is induced by the injection $\ga \subset E_k$ and the second is the arrow that sends an $E_k$-torsor $Z \rightarrow X$ to the $\ga$-torsor $Z \rightarrow Z/\ga=Y$. One can easily check that its image in $H^1(Y,\ga)=\mathcal{O}(G)$ is the one-dimensional subspace generated by $\chi_k$. We deduce that $H$-torsors over $Y$ that may be lifted to an $E_k$-torsor over $X$ for some $k\in\Z$ correspond, inside $H^1(Y,\mathbb{G}_a)=\cal O(G)$, to the union of the one-dimensional subspaces generated by the different $\chi_k\in\cal O(G)$. In particular, there exist $H$-torsors $Z\rightarrow Y$ such that the composite $Z\rightarrow Y\rightarrow X$ is not a torsor under an extension of $G$ by $H$.
\end{exa}

This example  leads  to a more general construction for towers of $\ga$-torsors over curves of genus $\geq 2$.

\begin{exa}\label{ex G=H=Ga}
Let $X/K$ be a smooth projective curve, of genus $g \geq 2$. Then, the $K$-vector space $H^1(X,\mathcal O_X)$ is $g$-dimensional. Let $Y\to X$ be a non-trivial $\ga$-torsor, whose class in $H^1(X,\mathcal \ga)$ we denote by $y$. Using the correspondence between $\ga$-torsors over $X$ and extensions of vector bundles of $\mathcal O_X$ by itself (both objects are classified by the group $H^1(X,\ga)=H^1(X,\cal O_X)$ by étale descent), $Y$ corresponds to an extension
\[\mathcal E:\qquad 0 \to \mathcal O_X \xrightarrow{s} E \xrightarrow{\pi}  \mathcal O_X \to 0.\]
More precisely, we have
\[Y = \spec\left( \varinjlim_n\, \Sym^n E \right),\]
where the transition morphisms $\Sym^n E \rightarrow \Sym^{n+1} E$ are given by mulitiplication by $s$, and hence
\[H^1(Y,\mathcal O_Y)=H^1(X,\varinjlim_n\, \Sym^n E) = \varinjlim_n\, H^1(X, \Sym^n E).\]
Now, to compute this direct limit, one can use the symmetric powers of $\mathcal{E}$:
\[\Sym^n \mathcal E:\qquad 0 \to \mathcal \Sym^{n-1}  E \xrightarrow{\times s} \Sym^{n} E \xrightarrow{\pi^n}  \mathcal O_X \to 0,\]
where $$\pi^n(e_1 \otimes \ldots \otimes e_n):=\pi(e_1) \ldots \pi(e_n).$$  Denote by $y^n\in H^1(X, \Sym^n E) $ the class of $\Sym^n \mathcal E$. We have a commutative diagram of extensions
\[\xymatrix{ 0 \ar[r] & \Sym^n(E)\ar[r]^{\times s} \ar[d]^{\pi^n} & \Sym^{n+1} (E)\ar[r]^-{\pi^{n+1}} \ar[d]^g & \mathcal O_E \ar[d]^{ \times (n+1)} \ar[r] & 0 \\ 0 \ar[r] &\mathcal O_E  \ar[r]^s & E\ar[r]^\pi& \mathcal O_E  \ar[r] & 0,}\]
where $g$ is given by the formula
\[g(e_0 \ldots e_{n})= \sum_0^{n}\pi(e_0) \ldots \widehat {\pi(e_i)} \ldots \pi(e_{n}) e_i,\]
where $\widehat{\phantom{a}}$ denotes an omitted variable.
Since $E$ is non-split, and since $(n+1) \in K^\times$, we get that (the class of) $\Sym^n \mathcal E$ does not belong to the image of $ H^1(X, \Sym^{n-1} E) \to  H^1(X, \Sym^{n} E)$; in particular, it is non-split.

Now, the cohomology exact sequence associated to $\Sym^n\cal E$ gives:
\[ K \to H^1(X,\Sym^{n-1} E) \xrightarrow{s_n} H^1(X,\Sym^{n} E) \to  H^1(X,\mathcal O_X) \to 0,\]
where the image of the leftmost arrow is precisely the subspace generated by the class of $\Sym^n \mathcal E$. This tells us that, if we set $V_n:=H^1(X, \Sym^n E) /\langle y^n\rangle$, we have an exact sequence
\[0\to V_n\to V_{n+1}\to H^1(X,\mathcal O_X) /\langle y\rangle \to 0.\]
Since $\dim(H^1(X,\mathcal O_X))=g\geq 2$, we see that the direct limit of the $V_n$'s has infinite dimension, so that the same holds for $H^1(Y,\cal O_Y)=\varinjlim_n\,H^1(X, \Sym^n E)$.

Assume now, that for every $\ga$-torsor $Z\to Y$, we can find  an extension of $X$-group schemes
\[ 1 \to \ga \to \Gamma \to \ga \to 1,\]
such that $Z \to X$ can be equipped with the structure of a $\Gamma$-torsor. Since $K$ has characteristic zero, $\Gamma$ is the affine space of a vector bundle $F$ over $X$, fitting into an extension of vector bundles over $X$
\[\mathcal F:\qquad 0 \to \mathcal O_X \xrightarrow{s} F \xrightarrow{\pi}  \mathcal O_X \to 0.\]
Using the same computation as above, we get that $H^1(X,F)$ is $(2g-1)$-dimensional. Thus, the moduli space of torsors under extensions of $\ga$ by itself is $(3g-1)$-dimensional. This contradicts the fact that $H^1(Y,\cal O_Y)$ is infinite-dimensional.
\end{exa}

\begin{rem}\label{rem affine et unip}
Note that in Examples \ref{ex G=aff H=Ga} and \ref{ex G=H=Ga} the base scheme $X$ is always proper. On the opposite side, when the base is affine, we get a particular case where the answer to Question \ref{ques} is positive with $H$ unipotent as follows:

Assume that $H$ is a \emph{split} unipotent group, that $G$ is linear and that $X$ is \emph{affine}. Let $Y$ be a $G$-torsor over $X$ and let $Z$ be an $H$-torsor avec $Y$. Then $Y$ is affine, and hence $H^1(Y,\mathbb{G}_{\mathrm{a}})=0$. We deduce that $H^1(Y,H)$ is trivial, so that $Z=Y \times H$. In particular, $Z$ is a $(G\times H)$-torsor over $X$.
\end{rem}

In contrast with the last remark, if $G$ is not linear, we can also provide an example in which the base scheme $X$ is not proper over $K$.

\begin{exa}\label{ex G=A H=Ga}
Let $G=A$ be an abelian variety, $H=\ga$ and let $X$ be the spectrum of a function field $L$ over $K$. Then $\ext(A,\ga)\simeq H^1(A,\cal O_A)$ by \cite[VII.17, Thm.~7]{SerreGpsAlg}, which is a $K$-vector space for group extensions over $K$ and an $L$-vector space of the same dimension if we do the corresponding base change (and the restriction arrow is the obvious injection). This tells us immediately that there are extensions of $A$ by $\ga$ over $L$ that do not come from extensions over $K$. In particular, these extensions are towers of torsors over $L$ that cannot have a torsor structure under an extension defined over $K$ (if an extension were a torsor under another extension, they would have the same underlying variety and hence define the same element in $H^1(A,\cal O_A)$). Obviously, these extensions can be built over a suitable (smooth affine) $K$-scheme with function field $L$, if one wants $X$ to be more than just a single point.
\end{exa}

\subsection{Examples where $H$ is semisimple}

We finish this section with examples in which $H$ is semisimple. This completes the study of all cases in Table \ref{table}.

\begin{exa}\label{ex G=aff H=PGLn}
Let $H=\pgl_n$ with $n\geq 2$ and let $G$ be either $\ga^m$, $\gm^m$, $\pgl_m$, or an abelian variety $A$. Consider an $H$-torsor $Z\to G$, and the trivial $G$-torsor $G\to \mathrm{Spec}(K)$ below it. If Question \ref{ques} had a positive answer for this tower, then $Z$ would be an $E$-torsor for some extension $E$ of $G$ by $H$. However, in all four cases for $G$ (assuming $n\gg m$ if $G=\ga^m$ and assuming $K$ algebraically closed if $G=A$) we have that the only possible extension is the direct product $E=G\times H$. Indeed, the first three cases are  elementary results from the theory of linear groups, and the case $G=A$ comes from \cite[Prop.~3.1.1]{BrionSamuelUma}. This implies in particular that the class in $H^1(G,H)$ of $Z\to G$ must come from $H^1(K,H)$. Thus, any class in $H^1(G,H)$ which does not come from $H^1(K,H)$ gives a negative answer to Question \ref{ques}. Now, recall that classes in $H^1(G,H)$ classify Azumaya algebras over $G$, which correspond also to classes in the Brauer group $\br(G)$ (cf.~for instance \cite[Thm.~3.3.2]{CTSkorBr}).

Assume that $G=\ga^m$ and that $\br(K)\neq 0$. Then by \cite[Prop.~2]{OjSr}, there exist non-constant Azumaya algebras over $\ga^2$. These correspond to elements in $H^1(G,H)$ that do not come from $H^1(K,H)$.

Assume that $G=\gm^m$. Then a simple computation using residue maps with respect to the irreducible divisors in $\bb P^m\smallsetminus\gm^m$ (cf.~for instance \cite[Thm.~3.7.2]{CTSkorBr}) tells us that $\br(G)/\br(K)\neq 0$. A class in $\br(G)\smallsetminus\br(K)$ corresponds then to a class in $H^1(G,H)$ which does not come from $H^1(K,H)$.

Assume that $G=\pgl_m$ and that $H^1(K,\Z/m\Z)\neq 0$. Since the subgroup of algebraic classes in $\br(G)/\br(K)$ is isomorphic to $H^1(K,\Z/m\Z)$ (cf.~\cite[Lem.~6.9(iii)]{Sansuc81}), one can find non-constant classes as well in this case.

Finally, assume that $G=A$ and that $K$ is algebraically closed. Then it is well-known that $\br(A)/\br(K)$ is non-trivial in general (cf.~\cite[p.~182]{BerkovichBrauer}). We conclude as before.
\end{exa}

\begin{rem}
Given that all the examples above use the adjoint group $H=\pgl_n$, one could wonder whether Question \ref{ques} has a positive answer when $H$ is semi-simple and simply connected. This question remains open.
\end{rem}

\begin{appendix}
\section{An elementary proof of Rosenlicht's Lemma}

We prove the following lemma, due to Rosenlicht in the case of an algebraically closed field (cf.~\cite{Rosenlicht}). 

\begin{lem}\label{RosenLem}
Let $H$ be a semi-abelian variety over a field $K$. Let $V$ and $W$ be geometrically integral $K$-varieties. Then, the following holds.
\begin{enumerate}
    \item The abelian group $H(W)/H(K)$ is finitely generated and free.
    \item If $K$ is separably closed, the sequence
\[ 0 \to H(K) \xrightarrow{h \mapsto (h,-h)}  H(V) \times H(W) \xrightarrow{\pi_V^*+\pi_W^*} H(V \times_K W) \to 0,\]
is exact, where
\begin{align*}
\pi_V^*:  H(V)  &\to  H(V \times_K W)\\
\quad h &\mapsto h \circ \pi_V.
\end{align*}
\end{enumerate}
\end{lem}

\begin{proof}
In both statements, $W$ and $V$ can be replaced by a non-empty open subvariety. In particular, by generic smoothness, we can thus assume that $V$ and $W$ are smooth over $K$.\\

Let us prove the first assertion. Denoting by $\bar K$ a separable closure of $K$, the natural arrow $$H(W)/ H(K) \to H(\bar W)/ H(\bar K)$$ is injective, so that we may assume $K= \bar K$. By definition, there is an exact sequence
\[1\to T\to H\xrightarrow{\pi} A\to 1,\]
with $T$ a torus and $A$ an abelian variety. Since $K$ is separably closed and $T$ is smooth, the snake lemma gives an exact sequence
\[1\to T(W)/T(K)\to H(W)/H(K)\xrightarrow{\pi} A(W)/A(K).\]
It will suffice then to treat the cases $H=A$, or $H=\gm$.

Up to shrinking $W$, we can assume there is a smooth $K$-morphism $W\to U$ of relative dimension one and with geometrically connected fibers, where $U=D(f)$ is a principal open subset of some affine space. Denote by $K'=K(U)$ the field of functions of $U$, and set $W':=W \times_U K'$. Then $W'$ is a smooth $K'$-curve, and there is an exact sequence 
\[ 0 \to H(U) / H(K) \to H(W)/H(K) \to   H(W')/H(K').\]
Thus, the problem is further reduced to two particular cases: $W=D(f)$ is a principal open subset of some affine space, or $W$ is a smooth curve over $K$. The first case is trivial for abelian varieties (a morphism from a rational variety to an abelian variety is constant). For $\gm$, it is dealt with by a straightforward direct computation. It remains to treat the case of a smooth affine curve $W/K$. The case $H=\gm$ is once again a straightforward computation, using the smooth proper curve $C$ compactifying $W$ and the fact that $H(C)=H(K)$. For $H$ an abelian variety, using \cite[Thm.~6.1]{MilneAV}, the statement is equivalent to $\hom_\mathrm{gp}(\mathrm{Jac(C)},H)$ being a free abelian group of finite rank, which holds by \cite[Thm.~10.15]{MilneAV}.\\

In order to establish the second assertion, we only need to check the surjectivity of $\pi_V^*+\pi_W^*$. Pick rational points $v_0 \in V(K)$ and $w_0 \in W(K)$, which exist since $K$ is separably closed and $V,W$ are smooth over $K$. For $f\in  H(V \times_K W)$, set
\[\tilde f(v,w):=f(v,w)-f(v_0,w)-f(v,w_0)+f(v_0,w_0).\]
Then, the composite
\[\pi \circ \tilde f: V \times_K W \to A\]
vanishes on $\{v_0 \} \times_K W$ and on $V \times_K \{w_0 \}$. Using \cite[Thm.~3.4]{MilneAV}, we get $\pi \circ \tilde f=0$. In other words, $\tilde f$ takes values in $T$. Thus, in order to conclude, it suffices to prove the exactness when $H=\gm$. Fix $f\in H(V\times_K W)=\mathcal O_{V \times_K W} ^\times$. Replacing $f$ by $(v,w)\mapsto f(v,w)f(v_0,w)^{-1}$, we may assume that $f=1$ on  $\{v_0 \} \times_K W$. To conclude, we have to prove that $f$ factors through the projection $\pi_V:V \times W \to V $, i.e. that $f$ does not depend on $W$. Using that the smooth $K$-variety $W$ is covered by smooth $K$-curves (for instance, by Bertini's Theorem), we easily reduce to the case where $W$ is a curve.

If $W$ is an open subset of $\gm$, this is once again a straightforward computation. In general, for a given $v_1 \in  V(K)$, set 
\begin{align*}
g_1: W &\to \gm,\\ w &\mapsto f(v_1,w).
\end{align*}
We have to show that $g_1$ is constant. Assume it is not. Then, it is quasi-finite of degree $d\geq 1$ over its image. Up to shrinking $W$, we may assume that $g_1$ is a composite arrow $W\to W'\to U\subset\gm$ with $W\to W'$ purely inseparable and $W'\to U$ finite and \'etale.

Assume first that $W=W'$. Denote by $\tilde W \to  U$ the Galois closure of $g_1$. There is a commutative diagram
\[\xymatrix{ \mathcal O_{V \times W} ^\times \ar[r]^{\rho_1} \ar[d]^{N} & {\mathcal O^\times_W }  \ar[d]^{N} \\{\mathcal O^\times_{V \times U}} \ar[r]^{\rho_1} & {\mathcal O^\times_U },}\] 
where $N$ is the (multiplicative) norm with respect to the finite étale morphism $g_1$ and $\rho_1$ denotes the restrictions to the fiber above $v_1$ (in particular, it maps $f$ to $g_1$). We claim that $N(f)\in \mathcal O_{V \times_K U} ^\times$ is trivial on $\{v_0 \} \times U$. Indeed, we have $N(f)=f_1 f_2 \ldots f_d,$ where $f=f_1,f_2,\ldots,f_d$ are the images of $f$ with respect to the different embeddings of $\cal O_{V\times_K W}$ in $\cal O_{V\times_K \tilde W}$. These are trivial on  $\{v_0 \} \times \tilde W$, whence the claim. Since we know the conclusion of the Lemma for $W=U$, we get that $N(f)\in \mathcal O_{ V \times U}^\times$ does not depend on $U$. Using commutativity of the diagram above, we compute: $$\rho_1(N(f))=N(g_1)=g_1^d. $$ Thus, $g_1^d$  is constant. Hence so is $g_1$, which finishes the proof when $W=W'$. In general, for a finite purely inseparable morphism of degree $p^r$, the norm is given by $N(x)=x^{p^r}$, so that a straightforward variant of the proof above applies.
\end{proof}

\begin{rem}
The second statement of Lemma \ref{RosenLem}, over a non-separably closed $K$, is false in general. Indeed, surjectivity fails when  $V=W$ is a non-trivial $H$-torsor. This is essentially the only counterexample, as surjectivity  holds whenever $H^1(K,H)=0$ (e.g. for $H=\gm$).
\end{rem}

\begin{rem}
When $H$ is a torus, the proof of the second statement of Lemma \ref{RosenLem} that we provide above uses affine geometry, combined with a norm argument. In this sense, it is an ``inner'' proof. This is a more elementary approach than the use of a normal compactification of $W$ in Rosenlicht's original proof. 
\end{rem}

\end{appendix}

Mathieu \textsc{Florence}, Sorbonne Universit\'e and Universit\'e Paris Cit\'e, CNRS, IMJ-PRG, F-75005 Paris, France.\\
Email address: \texttt{mflorence@imj-prg.fr}\\

Diego \textsc{Izquierdo}, Centre de Math\'ematiques Laurent Schwartz, \'Ecole Polytechnique, Institut Polytechnique de Paris, 91128 Palaiseau Cedex, France\\
Email address: \texttt{diego.izquierdo@polytechnique.edu}\\

Giancarlo \textsc{Lucchini Arteche}, Departamento de Matem\'aticas, Facultad de Ciencias, Universidad de Chile, Las Palmeras 3425, \~Nu\~noa, Santiago, Chile.\\
Email address: \texttt{luco@uchile.cl}


\begin{thebibliography}{ABCD00}

\bibitem[BDLM20]{McFaddin et cie}
{\bf M.~R.~Ballard, A.~Duncan, A.~Lamarche, P.~K.~McFaddin.}
Separable algebras and coflasque resolutions. arXiv:2006.06876

\bibitem[Ber72]{BerkovichBrauer}
{\bf V.~G.~Berkovich.} The Brauer group of abelian varieties. {\it Funktsional.~Anal.~i Prilozhen.} 6(3), 10--15, 1972 (Russian). English translation: {\it Funct.~Anal.~Appl. 6(3)}, 180--184, 1972.

\bibitem[BD13]{BD}
{\bf M.~Borovoi, C.~Demarche.} Manin obstruction to strong approximation for homogeneous spaces. {\it Comment.~Math.~Helv.} 88, 1--54, 2013.

\bibitem[BLR90]{NeronModels}
{\bf S.~Bosch, W.~L{\"u}tkebohmert, M.~Raynaud.} {\it N\'eron models}. Ergebnisse der Mathematik und ihrer Grenzgebiete (3), No.~{\bf 21}, Springer-Verlag, Berlin, 1990.

\bibitem[Bri12]{BrionHB}
{\bf M.~Brion.} Homogeneous bundles over abelian varieties. {\it J.~Ramanujan Math.~Soc.} 27, 91--118, 2012.

\bibitem[Bri13]{BrionPHB}
{\bf M.~Brion.} Homogeneous projective bundles over abelian varieties. {\it Algebra Number Theory} 7, 2475--2510, 2013.

\bibitem[Bri18]{BrionLin}
{\bf M.~Brion.} Linearization of algebraic group actions. {\it Handbook of group actions. Vol.~IV}, 291--340, Adv.~Lect.~Math. (ALM), No.~{\bf 41}, Int.~Press, Somerville, MA, 2018.

\bibitem[Bri20]{BrionHB3}
{\bf M.~Brion.} Homogeneous vector bundles over abelian varieties via representation theory, {\it Represent.~Theory} 20, 85--114, 2020.

\bibitem[Bri22]{BrionHom}
{\bf M.~Brion.} Homomorphisms of algebraic groups: representability and rigidity. {\it Michigan Math.~J.} 72 (Special issue in honor of Gopal Prasad), 51--76, 2022. 

\bibitem[BSU13]{BrionSamuelUma}
{\bf M.~Brion, P.~Samuel, V.~Uma.} {\it Lectures on the structure of algebraic groups and geometric applications}. CMI Lecture Series in Mathematics, No.~{\bf 1}, Hindustan Book Agency, 2013.

\bibitem[CTS21]{CTSkorBr}
{\bf J.-L.~Colliot-Th{\'e}l{\`e}ne, A.~N.~Skorobogatov.} {\it The Brauer-Grothendieck group.}
Ergebnisse der Mathematik und ihrer Grenzgebiete (3), No.~{\bf 71},
Springer, Cham, 2021.

\bibitem[Con02]{ConradChevalley}
{\bf B.~Conrad.} A modern proof of Chevalley's theorem on algebraic groups.
{\it J. Ramanujan Math. Soc.} 17, 1--18, 2002.

\bibitem[Dem65]{Demazure}
{\bf M.~Demazure.} Sch\'emas en groupes r\'eductifs. {\it Bull.~Soc.~Math.~France} 93, 369--413, 1965.

\bibitem[Gir71]{Giraud}
{\bf J.~Giraud.} {\it Cohomologie non ab\'elienne}. Die Grundlehren der ma\-the\-ma\-tischen Wissenschaften, No.~{\bf 179}. Springer-Verlag, Berlin-New York, 1971.

\bibitem[Gro68]{GrothendieckBrauerIII}
{\bf A.~Grothendieck.} Le groupe de {B}rauer. {III}. {E}xemples et compl\'ements. {\it Dix {E}xpos\'es sur la {C}ohomologie des {S}ch\'emas}, 88--188, North-Holland, Amsterdam, 1968.

\bibitem[HS05]{HS}
{\bf D.~Harari, A.~Skorobogatov.} Non-abelian descent and the arithmetic of Enriques surfaces. {\it Int.~Math.~Res.~Not.~IMRN} 52, 3203--3228, 2005.

\bibitem[ILA21]{ILA}
{\bf D.~Izquierdo, G.~Lucchini Arteche.} Local-global principles for homogeneous spaces over some two-dimensional geometric global fields. {\it J. Reine Angew. Math.} 781, 165--186, 2021.

\bibitem[Mil80]{Milne}
{\bf J.~S.~Milne.} {\it \'{E}tale cohomology}. Princeton Mathematical Series, No.~{\bf 33}. Princeton University Press, Princeton, N.J., 1980.

\bibitem[Mil86]{MilneAV}
{\bf J.~S.~Milne.} Abelian varieties. {\it Arithmetic geometry (Storrs, Conn., 1984)}, 103--150, Springer, New York, 1986.

\bibitem[OS71]{OjSr}
{\bf M.~Ojanguren, M.~R.~Sridharan.} Cancellation of Azumaya algebras. {\it J.~Algebra} 18, 501--505, 1971. 

\bibitem[Ray70]{raynaud}
{\bf M.~Raynaud.} {\it Faisceaux amples sur les schémas en groupes et les espaces homogènes}. Lecture Notes in Mathematics, No.~{\bf 119}. Springer-Verlag, Berlin, Heidelberg, 1970.

\bibitem[Ros61]{Rosenlicht}
{\bf M.~Rosenlicht.} Toroidal Algebraic Groups. {\it Proc.~Am.~Math.~Soc.} 12, 984--988, 1961.

\bibitem[San81]{Sansuc81}
{\bf J.-J.~Sansuc.} Groupe de {B}rauer et arithm\'etique des groupes alg\'ebriques lin\'eaires sur un corps de nombres. {\it J. Reine Angew. Math.} 327, 12--80, 1981.

\bibitem[Ser75]{SerreGpsAlg}
{\bf J.-P.~Serre.} {\it Groupes alg\'ebriques et corps de classes}. Publications de l'Institut de Math\'ematique de l'Universit\'e de Nancago, No.~{\bf VII}. Actualit\'es Scientifiques et Industrielles, No.~{\bf 1264}. Hermann, Paris, deuxi\`eme \'edition, 1975.

\bibitem[SGA4]{SGA4}
{\bf M.~Artin, A.~Grothendieck, J.-L.~Verdier (eds.).}
{\it Théorie des topos et cohomologie étale des schémas. Tome 3. Séminaire de géométrie algébrique du Bois-Marie 1963–1964 (SGA 4).} Lecture Notes in Mathematics, No.~{\bf 305}. {\it Springer-Verlag, Berlin-Heidelberg-New York}, 1973.

\bibitem[SGA7]{SGA7}
{\bf A. Grothendieck (ed.).} {\it Groupes de monodromie en g\'eom\'etrie alg\'ebrique. I. S\'eminaire de G\'eom\'etrie Alg\'ebrique du Bois-Marie 1967--1969 (SGA 7 I).} Lecture Notes in Mathematics, No.~{\bf 288}. {\it Springer-Verlag, Berlin-New York}, 1972.

\bibitem[Sha86]{Shatz86}
{\bf S.-S.~Shatz.} Group Schemes, Formal Groups, and p-Divisible Groups. {\it Arithmetic geometry (Storrs, Conn., 1984)}, 29--78, Springer, New York, NY, 1986.

\bibitem[SP18]{SP}
{\bf The Stacks Project Authors}, Stacks Project, 2018.\\
Available at {\tt https://stacks.math.columbia.edu}.

\end{thebibliography}
\end{document}